\documentclass[11pt,reqno]{amsart}

\usepackage{amscd,amssymb,amsmath,amsthm}
\usepackage[arrow,matrix]{xy}
\usepackage{graphicx}
\usepackage{epstopdf}
\usepackage{color}
\usepackage{tikz}
\usetikzlibrary{positioning,matrix,arrows,calc}
\usepackage{cite}
\newtheorem{thm}{Theorem}
\newtheorem{defn}{Definition}

\newtheorem{pro}{Proposition}

\numberwithin{equation}{section} 

\begin{document}

\title[On Volterra quadratic stochastic operators of a two-sex population]{On Volterra quadratic stochastic operators of a two-sex population on $S^1\times S^1$}

\author{ O. Castanos, U.U. Jamilov, U.A. Rozikov}

\address{ O. Castanos\\ Department of Mathematics, University of California, USA.}
\email{oscar$_-$castanos@mail.fresnostate.edu}

 \address{ U.\ U.\ Jamilov\\ Institute of Mathematics, Academy of Sciences of Uzbekistan,
81, Mirzo-Ulugbek str., 100170, Tashkent, Uzbekistan.}
\email{jamilovu@yandex.ru}

\address{ U.\ A.\ Rozikov\\ Institute of Mathematics, Academy of Sciences of Uzbekistan,
81, Mirzo-Ulugbek str., 100170, Tashkent, Uzbekistan.}
\email{rozikovu@yandex.ru}

\begin{abstract}
 We consider a four-parametric $(a, b, \alpha, \beta)$ family of Volterra quadratic stochastic
 operators for a bisexual population (i.e., each organism of the population
 must belong either to the female sex  or the male sex). We show that independently on parameters each such operator
 has at least two fixed points. Moreover, under some conditions on parameters the operator has infinitely many (continuum)
 fixed points. Choosing parameters, numerically we show that a fixed point may be any type:
 attracting, repelling, saddle and non-hyperbolic.
 We separate five subfamilies of quadratic operators and show that each operator of these
 subfamilies is regular, i.e. any trajectory constructed by the operator converges to a fixed point.
\end{abstract}



\maketitle

\section{Introduction}

Quadratic stochastic operators frequently arise in many models
of mathematical genetics, namely theory of heredity. The main problem is to describe the asymptotic
behavior of the trajectories of quadratic stochastic operators.
In this paper we consider a family of quadratic stochastic operators
of bisexual populations. For motivations of consideration and results related to dynamical systems of quadratic
stochastic operators see \cite{BlJaSc},\cite{GGJ,GaJa,GSJ,GZM,GR1,GR2,GaMuRo,K,LaRo,Lyu,RS,RoZa2,RoZa12,RoZh,RoZh3},\cite{ZhRo} and references therein.

 Let us give necessary definitions.

 Let  $E=\{1,2,\dots,m\}$. By the $(m-1)$-simplex we mean the set of all probability measures on $E$
 which is given by
\[
S^{m-1}=\{\mathbf{x}=(x_1,\dots,x_m)\in \mathbb{R}^m: x_i\geq 0, \ \sum^m_{i=1}x_i=1 \}.
\]

Following \cite{Lyu}, we construct the evolution
operator of a bisexual population. Suppose that the set of females can be partitioned
into finitely many different types indexed by $\{1,2,\dots,n\}$ and,
similarly, that the male types are indexed by $\{1,2,\dots,\nu\}$.
 The population is described by its state vector
 $$(\mathbf{x},\mathbf{y})=(x_1,x_2,\dots, x_n;\, y_1, y_2,\dots,y_\nu) \in S^{n-1}\times S^{\nu-1}.$$

    Denote $S=S^{n-1}\times S^{\nu-1}$. One says that the partition into
types is hereditary if for each possible state $\mathbf{z}=(\mathbf{x},\mathbf{y})\in S$
describing the current generation, the state
$\mathbf{z}'=(\mathbf{x}',\mathbf{y}')\in S$ is uniquely defined describing
the next generation. This means that the association $\mathbf{z}\mapsto \mathbf{z}'$ defines
a map $W \colon S\to S$, called the evolution operator.

    For any point $\mathbf{z}^{(0)}\in S$ the sequence $\mathbf{z}^{(t)}=W(\mathbf{z}^{(t-1)}),
\ t=1,2,\dots$, is called the trajectory of $\mathbf{z}^{(0)}$.
Denote by $\omega (\mathbf{z}^{0})$ the set of limiting points of the  trajectory $\mathbf{z}^{(t)}$.

 The definition of $W$ depends on the problem considered in the real life. Here we give a model which was
    firstly given in \cite{K}, see also \cite{Lyu}.
 Let $p_{ij,k}^{(f)}$ and $p_{ij,l}^{(m)}$ be the coefficients of inheritance
 defined as the probability of a female offspring being
type $k$ and, respectively, a male offspring being type $l$,
when the parental pair is $i,j$ ($i,k=1,2,\dots,n$; and $j,l=1,2,\dots,\nu$).
We have
\begin{equation}\label{inhcoefg}
p_{ij,k}^{(f)}\geq 0 , \qquad \sum \limits_{k=1}^n p_{ij,k}^{(f)} =1; \qquad
\  p_{ij,l}^{(m)}\geq 0 , \qquad  \sum \limits_{l=1}^\nu p_{ij,l}^{(m)}=1.
\end{equation}

Using these coefficients we define the operator $W: \mathbf{z}=(\mathbf{x}, \mathbf{y})\in S\to \mathbf{z}'=W(\mathbf{z})=(\mathbf{x}',\mathbf{y}')\in S$ as
\begin{equation} \label{ksobp5}
W:\left\{\begin{aligned}
x'_k & =\sum \limits_{i,j=1}^{n,\nu} p_{ij,k}^{(f)}x_iy_j, && 1\leq k \leq n,\\
y'_l & = \sum \limits_{i,j=1}^{n,\nu} p_{ij,l}^{(m)}x_iy_j, && 1\leq l \leq \nu.
\end{aligned}\right.
\end{equation}
This operator is called quadratic stochastic operator of a bisexual population (QSOBP).

In \cite{RoZh3} the notion of Volterra QSOBP (VQSOBP) was introduced and some dynamical systems
generated by such operators were studied.

\begin{defn}[\cite{RoZh3}]
The evolution operator \eqref{ksobp5} is called a Volterra quadratic stochastic operator
of a bisexual population if the  coefficients of inheritance \eqref{inhcoefg} satisfy the conditions
\begin{equation}\label{coefvolt}
\begin{aligned}
p^{(f)}_{ij,k}=0& \ \ \text{if} \ \   k\notin\{i,j\}, \qquad  1\leq i,k \leq n, \, 1\leq j \leq \nu,\\
p^{(m)}_{ij,l}=0& \ \ \text{if} \ \   l\notin\{i,j\}, \qquad  \ 1\leq i \leq n, \, 1\leq j,l \leq \nu.
\end{aligned}
\end{equation}
\end{defn}
The biological interpretation of condition \eqref{coefvolt} is evident: the offspring $k$ repeats the
genotype of one of its parents $\{i,j\}$.

In \cite{RoZh3}  the canonical form of a VQSOBP is given.
Moreover, the description of the set of fixed points is reduced to the
description of the fixed points of Volterra type operators considered in \cite{GS}.
Some Lyapunov functions are constructed and used to obtain an upper bound for
the set of limiting points of the
trajectories. But many such operators were not studied yet.
Because in this study there are many difficulties arising by non-linearity,
many parameters and high dimensions. Therefore to be able to obtain
some complete results one has to consider particular cases, simplifying the above mentioned difficulties.

In this paper we consider VQSOBP in the case $S^1\times S^1$, i.e. we restrict ourselves to low dimensions and take
the operator of the following form

\begin{equation}\label{volt1}
\widetilde{W}:\left\{\begin{array}{llll}
x'_1=x_1y_1+ax_1y_2+(1-b)x_2y_1,\\
x'_2=x_2y_2+bx_2y_1+(1-a)x_1y_2,\\
y'_1=x_1y_1+\alpha x_1y_2+(1-\beta)x_2y_1,\\
y'_2=x_2y_2+\beta x_2y_1+(1-\alpha)x_1y_2,
\end{array}\right.
\end{equation}
where $a,b,\alpha,\beta\in[0,1]$.

In this paper we investigate dynamical systems generated. The paper organized as follows.
In Section 2 we show that independently on parameters each such operator
 has at least two fixed points. Under some conditions on parameters the operator has infinitely many
 fixed points. We show that a fixed point may be any type:
 attracting, repelling, saddle and non-hyperbolic.
 We separate five subfamilies of quadratic operators and show that each operator of these
 subfamilies is regular. The subfamilies are separately studies in Sections 3-6.
\section{Fixed points}
First we give the following proposition related to trivial cases.
\begin{pro} \label{obs}
\begin{itemize} \hfill
\item[i)] If $a=b=1, \ \ \alpha=\beta=0$ then operator \eqref{volt1} is the identity operator;
\item[ii)] If $a=b=0, \ \ \alpha=\beta=1$ then  $\widetilde{W}^2(\mathbf{x},\mathbf{y})=(\mathbf{x},\mathbf{y})$ for any $(\mathbf{x},\mathbf{y})\in S^1\times S^1 $.
\end{itemize}
\end{pro}
\begin{proof}
One can easily verify the above results through substitution.
\end{proof}

Below, we assume that the conditions of Proposition \ref{obs} are not fulfilled.

Since $(\mathbf{x},\mathbf{y})\in S^1\times S^1$, we have that $x_2=1-x_1$ and $y_2=1-y_1$
(we denote $x_1=x$, $y_1=y$).
One can rewrite the VQSOBP \eqref{volt1} in the following form

\begin{equation}\label{volt2}
W:\left\{\begin{array}{ll}
x'=(b-a)xy+ax+(1-b)y,\\
y'=(\beta-\alpha)xy+\alpha x+(1-\beta)y,\\
\end{array}\right.
\end{equation}
\vspace{.4cm}
where $a,b,\alpha,\beta\in[0,1]$.

Denote \[\widetilde{x}={(1-b)y\over1+(a-b)y-a}.\]

\begin{thm}\label{fixp}
The fixed points of VQSOBP \eqref{volt2} are as follows,
\begin{itemize}
\item[i)] $(0,0)$ and $(1,1)$ are always fixed points.
\item[ii)] If $\alpha(1-b)\neq \beta(1-a)$, then $(0,0)$ and $(1,1)$ are the only fixed points for \eqref{volt2};
\item[iii)] If $\alpha(1-b)=\beta(1-a)$ and $1+(a-b)y-a \neq 0$, then $(\widetilde{x}, y)$ are fixed points
for any $y\in[0,1]$.
\item[iv)] When $1+(a-b)y-a=0$, we have;

\begin{enumerate}
\item[a)] For $y=0$, $a=1$ and $\alpha=0$, $(x,0)$ are fixed points.
\item[b)] For $a=1$, $b=1$, $(x,\frac{\alpha x}{(\alpha-\beta)x+\beta})$ are fixed points. If $x=0$, $\beta=0$, then $(0,y)$ are fixed points.

\end{enumerate}

\end{itemize}
\end{thm}

\begin{proof} The equation $W(x,y)=(x,y)$ has the following form

\begin{equation}\label{eq1}
\left\{\begin{array}{ll}
x=(b-a)xy+ax+(1-b)y,\\
y=(\beta-\alpha)xy+\alpha x+(1-\beta)y,\\
\end{array}\right.
\end{equation}

From the first equation of \eqref{eq1} we have

\begin{equation}\label{eq1-1}
((a-b)y+1-a)x=(1-b)y.
\end{equation}

{\it Case: $(b-a)y=1-a$}. In this case from \eqref{eq1-1} we get $b=1$ or $y=0$.
If $y=0$ then we obtain $(1-a)x=0$.  Due to the assumption $a\neq 1$, one has $x=0$, that is $(0,0)$ is a fixed point. If $b=1$ then
$(1-a)y=1-a$, due to the assumption $a\neq 1$, we obtain $y=1$.
Under these assumptions from the second equation of \eqref{eq1} we get $\beta x-\beta=0$. If $\beta=0$ then
a point $(x,1)$ is a fixed point. If $\beta\neq 0$ then $x=1, y=1$ is the fixed point.

{\it Case: $(b-a)y\neq1-a$}. From the first equation of \eqref{eq1} one has

\begin{equation}\label{eq2}
\widetilde{x}=x={(1-b)y\over1+(a-b)y-a}
\end{equation}

and using it from the second equation after some algebra we obtain

\begin{equation}\label{eq3}
\big(\alpha(1-b)-\beta(1-a)\big) (y-1)y= 0.
\end{equation}

i) From the above equation we have a solution for $y=1$ and $y=0$, using $\widetilde{x}$, one has the fixed points $(0,0)$ and $(1,1)$ for all VQSOBP \eqref{volt2}.\\

ii) If $\alpha(1-b)-\beta(1-a)\neq 0$ then from \eqref{eq3} and \eqref{eq2} one easily has
that only $(0,0)$ and (1,1) are solutions of the system \eqref{eq1}.\\

iii)  If $\alpha(1-b)-\beta(1-a)= 0$, then it is easy to verify that for any $y\in [0,1]$
it holds $0\leq \widetilde{x}\leq 1$, i.e. the point $(\widetilde{x},y)$ is a fixed point of the
operator \eqref{volt2} for any $y\in[0,1]$. We have that for $y=\frac{a-1}{a-b}$, $\widetilde{x}$ is undefined. However, one can easily verify that $\frac{a-1}{a-b}\in [0,1]$ only at $0$ and $1$, so we needn't consider it for $\widetilde{x}$ since it is defined at these two $y$ values.\\

iv) We have that $1+(a-b)y-a=0$ for the cases $(a)$, $(b)$ and $y=\frac{a-1}{a-b}$, but as stated previously $\frac{a-1}{a-b}$ needn't be considered.

a) For $y=0$ and $a=1$ we have $x=x$ and $0=\alpha x$ we know $(0,0)$ is a fixed point for all systems of type \eqref{volt2}, so with $\alpha=0$ we have $(x,0)$ is a fixed point.

b) Let $a=b=1$, then $x=x$ and $y=(\beta-\alpha)xy+\alpha x+(1-\beta)y$. Solving for $y$, we get $y=\frac{\alpha x}{(\alpha-\beta)x-\beta}$, so $(x,\frac{\alpha x}{(\alpha-\beta)x-\beta})$ is a fixed point. We now consider when $(\alpha-\beta)x-\beta=0$, this is true when $\alpha=\beta=0$ but this would then satisfy conditions from Proposition \ref{obs}. The previous also isn't defined at $x=\frac{-\beta}{\alpha-\beta}$, this however is only defined on $[0,1]$ at $1$ when $\alpha=0$. So at $x=1$ and $\alpha=0$ we get $(1,1)$.  For $x=\beta=0$, one can easily verify we get the fixed point $(x,0)$.
\end{proof}

To study the types of fixed points consider the Jacobian matrix of the operator \eqref{volt2} at a fixed point $(x, y)$:
\[
J_W(x,y)=\left(\begin{array}{cc}
(b-a)y+a & (b-a)x+(1-b)\\
(\beta-\alpha)y+\alpha & (\beta-\alpha)x+(1-\beta)
\end{array}\right)
\]
We will list the eigenvalues associated to the Jacobian matrix at our fixed points. For the fixed point $(0,0)$ one has the following two eigenvalues
\begin{equation*}\label{eig}
\lambda_{1,2}={1+a-\beta\pm\sqrt{(\beta-1+a)^2+4\alpha(1-b)}\over 2}.
\end{equation*}

For the fixed point $(1,1)$, we have
\begin{equation*}\label{eig1}
\lambda_{1,2}=\frac{1-\alpha+b\pm \sqrt{(\alpha+b-1)^2+4\beta(1-a)}}{2}
\end{equation*}
Let
\[\gamma_1=(b-a)(1-\beta)+(\alpha-\beta)(1-b), \ \ \gamma_2=a(\beta-\alpha)+\alpha(a-b).\]

Then for Theorem \ref{fixp} (iii-iv) we have,

iii)
\begin{align*}
\lambda_{1,2}=&\frac{1}{2}((b-a)y+(\beta-\alpha)\widetilde{x}+1+a-\beta\\
  &\pm \sqrt{((a-b)y+(\alpha-\beta)\widetilde{x}+\beta-a-1)^2-4(\gamma_1y+\gamma_2\widetilde{x}+a(1-\beta)+\alpha(b-1))}\hspace{.05 cm})\\
\end{align*}

iv) (a) $(x,0)$
\begin{equation*}\label{eiga}
\lambda_1=1+(x-1)\beta,\text{ } \lambda_2=1
\end{equation*}

(b) $(x,\frac{\alpha x}{(\alpha-\beta)x+\beta})$
\begin{equation*}
\lambda_1=1,\text{ } \lambda_2=1-\beta+(\beta-\alpha)x
\end{equation*}

For $(0,y)$, we simply have that $\lambda_{1,2}=1$.

In the table below we show the possible pairs of fixed point types of $(0,0)$ and $(1,1)$. We will also give examples of the behaviors in the chart below where we write the rounded magnitude of the eigenvalues of our Jacobian matrix for our fixed points as ordered pairs,  $(|\lambda_1|,|\lambda_2|)$.

\begin{equation*}
\begin{array}{c|c|c|c}
(a,b,\alpha,\beta) & (|\lambda_1|,|\lambda_2|)_{(0,0)} & (|\lambda_1|,|\lambda_2|)_{(1,1)} & Type: (0,0), (1,1) \\\hline
 & & & \\
(0.67,0.97,0.896,0.908) & (0.713, 0.0487) & (0.836,0.238) & Attracting, Attracting\\
 & & & \\
(0.173,0.718,0.027,0.927) & (0.224,0.022) & (1.210,1.210) & Attracting, Repelling \\
 & & & \\
(0.487,0.329,0.0017,0.0675) & (0.935,0.484) & (1.521,0.193) & Attracting, Saddle \\
 & & & \\
(0.345,0.6244,0.829,0.185) & (1.185,0.025) & (0.777,0.0185) & Saddle, Attracting \\
 & & &\\
(0.422,0.786,0.584,0.024) & (1.148,0.025) & (1.422,0.220) & Saddle, Saddle\\
\end{array}\\
\end{equation*}

With the fixed points $(\widetilde{x},y)$, we have that the types vary along the curve. The types are either saddles, non-hyperbolic, or attracting. The behavior of the fixed points generally changes starting from $y=0$ and then as $y$ tends to $1$. We had the following behaviors

\[Saddle \rightarrow Attracting\]
\[Attracting \rightarrow Saddle\]

Which would imply at some point we have a non-hyperbolic point in the transition from the two behaviors. We also found that in some cases the fixed points remained saddle points or attracting points. We have that all fixed points from Theorem \ref{fixp} (iv) are non-hyperbolic, since they each have the eigenvalue 1.

\section{Case $a=b, \ \alpha=\beta$.}
In this conditions the operator \eqref{volt2} has following form

\begin{equation}\label{voltlin}
W:\left\{\begin{array}{ll}
x'=ax+(1-a)y,\\
y'=\alpha x+(1-\alpha)y,\\
\end{array}\right.
\end{equation}
\vspace{.4cm}
where $a,\alpha \in[0,1]$, that is the operator \eqref{voltlin} is a linear operator.

It is easy to check that the points $(x,x)\in [0,1]^2$ are fixed points of \eqref{voltlin}.
If $a=\alpha$ then one has that $W(x,y)=(x',x')$, that is the image of any point is a fixed point.

Let the set $\text{int}S=\{\mathbf{z}=(\mathbf{x},\mathbf{y})\in S: x_1x_2\dots x_m>0, \, y_1y_2\dots y_\nu>0 \}$ be the interior of the simplex $S$.

\begin{defn}  A continuous functional $\varphi:  S \to \mathbb{R}$ is called a Lyapunov function for a operator $W$ if  $\varphi(W (\mathbf{z}))\geq \varphi(\mathbf{z})$  for all $\mathbf{z}$  (or $\varphi(W (\mathbf{z}))\leq \varphi(\mathbf{z})$ for all $\mathbf{z}$).
\end{defn}

Suppose that $a>\alpha$ then from \eqref{voltlin} we obtain

\begin{equation}\label{lyaf}
x'-y'=(a-\alpha) (x-y)  \ \ \text{and} \ \ x'+y'=x+y +(a+\alpha)(x-y)
\end{equation}

If $x>y $ then due to $a-\alpha>0, \ \ a+\alpha>0$ one has $x'>y' $ and respectively
if $x<y $ then one has $x'<y'$. Consequently
using \eqref{lyaf} we have the functions
\begin{equation}\label{lyaf1}
\varphi(x,y)=x-y, \ \ \psi(x,y)=x+y
\end{equation}
are Lyapunov functions for the operator \eqref{voltlin}.
Suppose $x<y$ then from \eqref{lyaf1} and \eqref{voltlin} one has that
\begin{equation*}\label{lyaf2}
\varphi(x',y')\geq (a-\alpha) (x-y), \ \ \psi(x,y)\leq x+y
\end{equation*}
So, the function $\varphi(x,y)$ is increasing and bounded from above  and  the function $\psi(x,y)$ is decreasing and bounded from below it follows that existence of the $\lim\limits_{n\rightarrow\infty} \psi(x^{(n)},y^{(n)})$ and $\lim\limits_{n\rightarrow\infty} \varphi(x^{(n)},y^{(n)})$. Moreover using $|a-\alpha|<1$ we obtain that $\lim\limits_{n\rightarrow\infty} \varphi(x^{(n)},y^{(n)})=0$. Therefore there are $\lim\limits_{n\rightarrow\infty} x^{(n)}=x^*$ and $\lim\limits_{n\rightarrow\infty} y^{(n)}=y^*$ and $x^*=y^*$.

Assume that $x>y$ then from \eqref{lyaf1} and \eqref{voltlin} one has that
\begin{equation*}\label{lyaf3}
\varphi(x',y')\leq (a-\alpha) (x-y), \ \ \psi(x,y)\geq x+y
\end{equation*}
So, due to the fact that the function $\varphi(x,y)$ is decreasing and bounded from below  and  the function $\psi(x,y)$ is  increasing and bounded from above it follows that existence of the $\lim\limits_{n\rightarrow\infty} \psi(x^{(n)},y^{(n)})$ and $\lim\limits_{n\rightarrow\infty} \varphi(x^{(n)},y^{(n)})$. Moreover using $|a-\alpha|<1$ we obtain that $\lim\limits_{n\rightarrow\infty} \varphi(x^{(n)},y^{(n)})=0$. Therefore there are $\lim\limits_{n\rightarrow\infty} x^{(n)}=x^*$ and $\lim\limits_{n\rightarrow\infty} y^{(n)}=y^*$ and $x^*=y^*$.

Denote
\[\xi_n=x^{(n)}-y^{(n)}, \ \   \eta_n=x^{(n)}+y^{(n)}, \ \ n=0,1,\dots \]

Using \eqref{lyaf} one easily has that
\[\xi_{n+1}=(a-\alpha)\xi_n,  \ \  \eta_{n+1}=\eta_n+(a+\alpha)\xi_n, \ \ n=0,1,\dots\]

Since $\xi_n=(\eta_{n+1}-\eta_n)/(a+\alpha)$ it follows
\begin{equation}\label{difeq}
\eta_{n+2}-(1+a-\alpha)\eta_{n+1}+(a-\alpha)\eta_n=0.
\end{equation}

It is known that the characteristic polynomial of \eqref{difeq} has the form
\[\lambda^2-(1+a-\alpha)\lambda+(a-\alpha)=0\]
and its solution are $\lambda_1=1$ and $\lambda_2=a-\alpha$. Consequently the solution of the difference
equation  \eqref{difeq} is
\[\eta_n=A_1+A_2(a-\alpha)^n, \ \ n=0,1,\dots\]

From the system of equations
\[\eta_0=A_1+A_2=x^{(0)}+y^{(0)}, \ \ \eta_1=A_1+A_2(a-\alpha)=(a+\alpha)x^{(0)}+(2-(a+\alpha))y^{(0)} \]
one has
\[
A_1={2\over 1-a+\alpha}\big(\alpha x^{(0)}+(1-a)y^{(0)} \big), \ \ A_2={1-a-\alpha\over 1-a+\alpha}\big(x^{(0)}+y^{(0)}\big).
\]

Therefore we obtain
\[
\eta_n={2\alpha+(1-a-\alpha)(a-\alpha)^n\over 1-a+\alpha} x^{(0)}+{2(1-a)-(1-a-\alpha)(a-\alpha)^n\over 1-a+\alpha}y^{(0)}.
\]

Due to $0<a-\alpha<1$ one has
\[\lim\limits_{n\rightarrow\infty}\xi_{n+1}=(a-\alpha)\xi_n=0, \]
\[\lim\limits_{n\rightarrow\infty}\eta_n={2\alpha\over 1-a+\alpha} x^{(0)}+{2(1-a)\over 1-a+\alpha}y^{(0)}\]

Thus
\begin{equation}\label{xf}
x^*=y^*=\lim\limits_{n\rightarrow\infty} x^{(n)}=\lim\limits_{n\rightarrow\infty} y^{(n)}={\alpha\over 1-a+\alpha} x^{(0)}+{(1-a)\over 1-a+\alpha}y^{(0)}
\end{equation}

The case $a<\alpha$ can be considered in  similar manner. We have proved the following

\begin{pro}
If $a=b, \, \alpha=\beta$ then for any initial point $(x^{(0)},y^{(0)})$ the trajectory of operator \eqref{voltlin} converges to the fixed point $(x^*,y^*)$ given in (\ref{xf}).
\end{pro}

\section{Case $\alpha=\beta=0$, $a\neq b$.} In this case the operator \eqref{volt2} has the following form

\begin{equation}\label{2cordlin}
W:\left\{\begin{array}{ll}
x'=(b-a)xy+ax+(1-b)y,\\
y'=y,\\
\end{array}\right.
\end{equation}
\vspace{.4cm}
where $a,b \in[0,1]$.

Since $y^{(n)}=y^{(0)}, \, n=1,2,\dots$ from the first equation of
\eqref{2cordlin} one has
\[
x'=\big((b-a)y^{(0)}+a\big)x+(1-b)y^{(0)}
\]
and
\begin{equation}\label{eq2cord}
x^{(n+1)}=\big((b-a)y^{(0)}+a\big)^{n+1}x^{(0)}+(1-b)y^{(0)}\sum\limits_{k=0}^n \big((b-a)y^{(0)}+a\big)^{k}, \ \ n=0,1,2,\dots
\end{equation}

Suppose that $0<(b-a)y^{(0)}+a<1$ then from \eqref{eq2cord} one has
\[
\lim\limits_{n\rightarrow\infty} x^{(n)}={(1-b)y^{(0)}\over 1-a-(b-a)y^{(0)}}.
\]
Consequently, the trajectory of operator \eqref{2cordlin} converges to the point
\begin{equation}\label{xff}
(x^*,y^*)=\bigg({(1-b)y^{(0)}\over 1-a-(b-a)y^{(0)}}, y^{(0)}\bigg).
\end{equation}

Similar result can be shown when $-1<(b-a)y^{(0)}+a<0$.\\

\begin{pro}
If $\alpha=\beta=0, \, a\neq b$ then for any initial point $(x^{(0)},y^{(0)})$ the trajectory of operator \eqref{2cordlin} converges to the fixed point $(x^*,y^*)$ given in (\ref{xff}).
\end{pro}

CASE: $a=b=1$, $\alpha \neq \beta$ is similar to the previous case due to the symmetry of the system.

\section{Case $b=1, \, \alpha=0$.}
In this case the operator \eqref{volt2} has the following form

\begin{equation}\label{volcord}
W:\left\{\begin{array}{ll}
x'=(1-a)xy+ax,\\
y'=\beta xy+(1-\beta)y,\\
\end{array}\right.
\end{equation}
\vspace{.4cm}
where $a,\beta \in[0,1]$.

Since $(1,1)$ is a fixed point of the operator \eqref{volcord} for any $(x,y)\neq(1,1)$ one has
that
\[x'=x((1-a)y+a)< x \ \  \text{and} \ \ y'=y(\beta x+(1-\beta))< y. \]

Therefore we get $x^{(n+1)}< x^{(n)}$ and $y^{(n+1)}< y^{(n)}$ for $n=0,1,2,\dots$. Since the sequences
$\{x^{(n)}\}$ and $\{y^{(n)}\}$ are decreasing and bounded from below it follows that there are
\[ \lim\limits_{n\rightarrow\infty} x^{(n)} = x^* \ \ \text{and} \ \ \lim\limits_{n\rightarrow\infty} y^{(n)} = y^* \]

We claim that $x^*=0$ and $y^*=0$. Suppose on the contrary that $x^*>0$. Then
\[
1=\lim\limits_{n\rightarrow\infty} {x^{(n+1)}\over x^{(n)}} = \lim\limits_{n\rightarrow\infty}  \big((1-a) y^{(n)} +a\big)
\ \ \Rightarrow \ \ \lim\limits_{n\rightarrow\infty}  y^{(n)}=1
\]

The last contradicts $0<y^{(n+1)}< y^{(n)}<\dots<y^{(0)}<1$, and hence $x^*=0$.

The claim $y^*=0$ can be shown in a similar manner.

\begin{thm}
If $b=1, \, \alpha=0$ then the trajectory of operator \eqref{volcord} converges to the fixed point $(0,0)$ for any initial point $(x^{(0)},y^{(0)})\neq(1,1)$.
\end{thm}

\section{Case $a=\alpha, \, b=\beta$:}
In this case the operator \eqref{volt2} has following form

\begin{equation}\label{invcord}
W:\left\{\begin{array}{ll}
x'=(b-a)xy + ax + (1-b)y,\\
y'=(b-a) xy + a x + (1-b)y,\\
\end{array}\right.
\end{equation}
\vspace{.4cm}
where $a, b \in[0,1]$.

 It is clear that $x'=y'$ for any initial point $(x^{(0)},y^{(0)})$.
Therefore the following set is invariant with respect to the operator:
$$I=\{(x,y)\in [0,1]^2: x=y\}.$$
Thus for any  $x^0=(x^{(0)}, y^{(0)})$ we have that $x'=V(x^0)\in I$, i.e. at first step
the trajectory comes in the invariant set $I$ and therefore stays there. Consequently, it
will be sufficient to study the dynamics of the operator restricted on the invariant set $I$.
Then taking $x=y$ from the operator (\ref{invcord}) we get the following function
\begin{equation}\label{invop}
f(x)\equiv x'=(b-a)x^2+(1+a-b)x.
\end{equation}

The function \eqref{invop} has the following fixed points
\[x^*=0  \ \ \text{and} \ \ x^{**}=1. \]

Further we need some notations and results from the well known dynamical system generated by quadratic
function $F_\mu(x)=\mu x(1-x)$.

\begin{defn}[\cite{De}] Let $h: A\rightarrow A$ and $g: B\rightarrow B$ be two maps. $h$ and $g$ are said
to be topologically conjugate if there exists a homeomorphism $\varphi:A\rightarrow B$ such that
$\varphi\circ h = g \circ\varphi$. The homeomorphism $\varphi$ is called a topological conjugacy.
\end{defn}

\begin{thm}[\cite{Shar}]\label{shar}\hfill
\begin{itemize}
\item[1.] If $0<\mu\leq 1$ then $F_\mu(x)$ has the unique fixed point $x=0$ and
for any initial $x^{(0)}\in [0,1]$ the trajectory converges to this fixed point.
\item[2.] If $1<\mu\leq 3$ then  $F_\mu(x)$ has an attracting fixed point $p_\mu=\frac{\mu-1}{\mu}$
and repelling fixed point 0. Moreover $\lim\limits_{n\rightarrow \infty} F^n_\mu(x)=p_\mu$ for any $0<x<1$.
\end{itemize}
\end{thm}

On the real line $\mathbb{R}$ one can examine that there is a
topological conjugacy  $\varphi$ such that
$\varphi\circ f = F_\mu\circ\varphi$,
where
\begin{equation*}\label{conj}
\varphi(x)=\frac{a-b}{1+a-b}x, \   \
 \  \ F_\mu(x)=(1+a-b)x(1-x).
\end{equation*}

Thus using Theorem \ref{shar} one can show

\begin{thm}\hfill
\begin{itemize}
\item[1.] If $a=\alpha\leq b=\beta$, then the trajectory of operator \eqref{invop} converges to the fixed point $(0,0)$ for any initial point $(x^{(0)},y^{(0)})\neq(1,1)$.
\item[2.] If $a=\alpha>b=\beta$, then the trajectory of operator \eqref{invop} converges to the fixed point $(1,1)$ for any initial point $(x^{(0)},y^{(0)})\neq(0,0)$.
\end{itemize}
\end{thm}

{\bf Acknowledgements}.
The first author was supported by the National Science Foundation, grant number NSF HRD 1302873.

\end{document}